\documentclass{article}
\usepackage[utf8]{inputenc}
\usepackage[left=1in,top=1in,right=1in,bottom=1in]{geometry}
\usepackage[linktocpage=true]{hyperref}
\usepackage{setspace}
\usepackage{amssymb, amsmath, amsthm, graphicx,mathrsfs}
\usepackage{caption,cite}
\usepackage{mathtools,color}
\DeclarePairedDelimiter{\ceil}{\lceil}{\rceil}
\DeclarePairedDelimiter{\floor}{\lfloor}{\rfloor}

\newtheorem{thm}{Theorem}[section]
\newtheorem{cor}[thm]{Corollary}

\newtheorem{definition}[thm]{Definition}
\newtheorem*{rem}{Remark}
\newtheorem{lem}[thm]{Lemma}
\newtheorem{conj}[thm]{Conjecture}
\newtheorem{prop}[thm]{Proposition}
\newtheorem*{claim}{Claim}

\newcommand{\F}{\mathcal{F}}
\newcommand{\G}{\mathcal{G}}

\newcommand{\cH}{\mathcal{H}}
\renewcommand{\t}{\hat{t}}

\renewcommand{\l}{\left}
\renewcommand{\r}{\right}

\usepackage[affil-it]{authblk}

\title{On asymptotic local Tur\'an problems}
\date{Feburary 2023}
\author[1]{Peter Frankl \thanks{frankl.peter@renyi.hu}}
\author[2]{Jiaxi Nie\thanks{jiaxi\_nie@fudan.edu.cn} }

\affil[1]{Rényi Institute, Budapest, Hungary}
\affil[2]{Shanghai Center for Mathematical Sciences, Fudan University, Shanghai, China}

\begin{document}

\maketitle

\begin{abstract}
An $r$-uniform hypergraph has $(q,p)$-property if any set of $q$ vertices spans a complete sub-hypergraph on $p$ vertices. Let $t_r(n,q,p)$ be the minimum edge density of an $n$-vertex $r$-uniform hypergraph with {\em $(q,p)$-property} and let $t_r(q,p)=\lim_{n\to\infty}t_r(n,q,p)$. A disjoint union of $k$ complete hypergraphs has $(q,\ceil{q/k})$-property, which gives $t_r((q,\ceil{q/k}))\le 1/k^{r-1}$. The first author, Huang and Rödl showed that these constructions are the best asymptotically, that is, $\lim_{q\to\infty}t_r((q,\ceil{q/k}))=1/k^{r-1}$. They asked whether it is true for all real number $\gamma\ge1$ that $\lim_{q\to\infty}t_r((q,\ceil{q/\gamma}))=1/\floor{\gamma}^{r-1}$. In this paper, we give positive answers to this question for a small range of real numbers, and, on the other hand, provide new constructions that give negative answers for many other ranges. 
\end{abstract}

\section{Introduction}

A hypergraph $\cH$ is a collection of subsets of a finite set $V(\cH)$. Members of $V(\cH)$ are called vertices of $\cH$, and members of $\cH$ are called edges. An $r$-uniform hypergraph (or $r$-graph for short) is a hypergraph whose edges all have size $r$.  For integers $q\ge p\ge r\ge 2$, an $r$-uniform hypergraph $H$ has $(q,p)$-property if for any $A\in\binom{V(\cH)}{q}$ there exists $B\in\binom{A}{p}$ such that $\binom{B}{r}\subset \cH$. Let $T_r(n,q,p)$ be the minimum integer such that an $n$-vertex $r$-graph with $(q,p)$-property has $T_r(n,q,p)$ edges. By a simple averaging argument, Katona, Nemetz and Simonovits~\cite{katona1964graph} showed that $T_r(n,q,p)/\binom{n}{r}$ is monotone increasing. Therefore, the limits
\begin{equation*}
    t_r(q,p)=\lim_{n\to\infty}\frac{T_r(n,q,p)}{\binom{n}{r}}
\end{equation*}
exist. For the graph case, Erdös and Spencer~\cite{erdos1974probabilistic} showed that
\begin{equation}\label{equation:graph}
    t_2(q,p)=1/\l\lfloor\frac{q-1}{p-1}\r\rfloor.
\end{equation}
For general $r$, the first author and Stechkin~\cite{stechkin1981local} proved that
\begin{equation}\label{equation:dense}
    t_r(q,p)=1~~\text{if}~q\le\frac{r}{r-1}(p-1).
\end{equation}
In this paper, we focus on the asymptotic behavior of $t_r(q,p)$.

\begin{definition}
    For every real number $\gamma\ge1$, let
\begin{equation*}
    \t_r(\gamma)=\lim_{q\to\infty}t_r(q,\lceil q/\gamma \rceil).
\end{equation*}
\end{definition}

See the Appendix for the proof of the existence of the limits.

Clearly, $\t_r(\gamma)$ is monotone non-increasing. Equation (\ref{equation:graph}) implies
\begin{equation}\label{equation:graph_asymp}
    \t_2(\gamma)=1/\floor{\gamma}
\end{equation}
for every $\gamma\ge1$, and equation (\ref{equation:dense}) implies
\begin{equation}\label{equation:dense_asymp}
    \t_r(\gamma)=1
\end{equation}
for every $r\ge 2$ and $1\le \gamma<\frac{r}{r-1}$. In~\cite{frankl1984asymptotic} it is proved that $\t_3(2)=\frac{1}{4}$. This was recently generalized by the first author, Huang and Rödl~\cite{frankl2021local} who showed the following:
\begin{thm}[\cite{frankl2021local}]\label{theorem:Frankl-Huang-Rodl}
    For all integers $r\ge 2$, $k\ge 1$,
    \begin{equation}\label{equation:integer}
    \t_r(k)=\frac{1}{k^{r-1}}.
    \end{equation}
    For every integers $p\ge 3$,
    \begin{equation*}
    t_3(2p+1,p+1)=\frac{1}{4}.
    \end{equation*}
\end{thm}

They asked whether it is true that for every real number $\gamma\ge 1$ 
\begin{equation}\label{equation:mainquestion}
    \t_r(\gamma)=1/\floor{\gamma}^{r-1}.
\end{equation}

Towards (\ref{equation:mainquestion}), we first prove the following general lower bounds:
\begin{thm}\label{theorem:general_lower_bounds}
    For every integer $r\ge2$, and real number $\gamma\ge1$,
    \begin{equation*}
        \t_r(\gamma)\ge\frac{1}{\gamma^{r-1}}
    \end{equation*}
\end{thm}

The next theorem, together with some known results, gives better lower bounds confirming (\ref{equation:mainquestion}) for a small range of real numbers.
\begin{thm}\label{theorem:better_lower_bound}
    For integers $q\ge r\ge2$, and real number $\gamma\ge1$, 
    \begin{equation*}
        \t_r(\gamma)\ge t_r(q,\ceil{q/\gamma}).
    \end{equation*}
    As a result, if $\t_r(\alpha)=t_r(q,\ceil{q/\alpha})$ for some $\alpha\ge1$, then for every real number $\gamma$ with $\alpha\le\gamma<\frac{q}{\ceil{q/\alpha}-1}$,
    \begin{equation*}
        \t_r(\gamma)=\t_r(\alpha).
    \end{equation*}
\end{thm}



\begin{rem}
    Since $t_2(q,2)=\t_2(q-1)=\frac{1}{q-1}$ and $t_r(r,r)=\t_r(1)=1$, Theorem~\ref{theorem:better_lower_bound} implies (\ref{equation:graph_asymp}) and (\ref{equation:dense_asymp}).
\end{rem}

Also, since $t_3(7,4)=\t_3(2)=1/4$ by Theorem~\ref{theorem:Frankl-Huang-Rodl}, we have the following corollary:

\begin{cor}\label{corollary:7/3}
    For $2\le\gamma<7/3$,
    \begin{equation}\label{equation:7/3}
    \t_3(\gamma)=\frac{1}{4}.
    \end{equation}
\end{cor}

Recently Fang, Gao, Ma and Song~\cite{Fang2023Onlocal} have also proved (\ref{equation:7/3}).

Tur\'an~\cite{turan1969applications} proposed the following notorious conjectures:
\begin{conj}[\cite{turan1969applications}]\label{conjecture:Turan}
For every integers $k\ge 2$,
    \begin{equation*}
        t_3(k+1,3)=\frac{4}{k^2}.
    \end{equation*}
\end{conj}

This conjecture, if true, implies the following corollary:
\begin{cor}\label{corollary:turan}
If Conjecture~\ref{conjecture:Turan} is true, then for any integer $k\ge 2$ and $\gamma<\frac{k+1}{2}$,
\begin{equation*}
    \t_3(\gamma)\ge\frac{4}{k^2}.
\end{equation*}
In particular, for every integer $t\ge 1$ and $t\le\gamma<t+\frac{1}{2}$, since $\t_3(t)=t_3(2t+1,3)=1/t^2$,
\begin{equation*}
    \t_3(\gamma)=\frac{1}{t^2}.
\end{equation*}
\end{cor}


On the other hand, we find new constructions with edge density smaller than that of the disjoint union of cliques, which give negative answers to (\ref{equation:mainquestion}) for many ranges of real numbers.
\begin{thm}\label{theorem:tightcycle_construction}
    Let integers $r\ge3$, $k\ge 1$.
    \begin{itemize}
        \item[(1)] For every integer $m\ge 1$, and real number $k+\frac{m}{m+1}\le \gamma<k+\frac{m+1}{m+2}$,
        \begin{equation}\label{equation:upper_bounds}
            \t_r(\gamma)\le \l(k-1+(2m+1)\l((m+1)^r-m^r\r)^{-\frac{1}{r-1}}\r)^{-r+1}
        \end{equation}
        \item[(2)] For every integer $3\le m\le r$, and real number $k+\frac{1}{m-1}\le \gamma<k+\frac{1}{m-2}$,
        \begin{equation}\label{equation:upper_bounds2}
            \t_r(\gamma)\le \l(k-1+\l(\sum^m_{i=1}(-1)^{i+1}\binom{m}{i}\l(1-\frac{i}{m}\r)^r\r)^{-\frac{1}{r-1}}\r)^{-r+1}
        \end{equation}
    \end{itemize}
\end{thm}

Recently, Fang, Gao, Ma and Song~\cite{Fang2023Onlocal} also proved the upper bounds (\ref{equation:upper_bounds2}) for $k=1$ using the same construction. Moreover, they find constructions that give upper bounds better than (\ref{equation:upper_bounds}) when 
$r=3$ and $m\ge 2$.


The rest of this paper is organized as follows: 

In Section~\ref{section:general} we generalize ideas from~\cite{frankl2021local} to prove the general lower bounds, that is, Theorem~\ref{theorem:general_lower_bounds}. 

In Section~\ref{section:better}, we proved Theorem~\ref{theorem:better_lower_bound} which, together with some known results, improves the lower bounds of Theorem~\ref{theorem:general_lower_bounds}.

In Section~\ref{section:tight}, we develop a framework that generates constructions for upper bounds and then use it to prove Theorem~\ref{theorem:tightcycle_construction}. 

\section{General lower bounds}\label{section:general}

In this section, we prove Theorem~\ref{theorem:general_lower_bounds}. A {$(w,v)$-hole} of an $r$-graph $\cH$ is a set of vertices $W\subset V(\cH)$ such that $|W|=w$ and $v$ is the {\em clique number}, that is, the maximum integer such that there exists $A\in\binom{W}{v}$ such that $\binom{A}{r}\subset\cH$. For any $w\ge u>v$, the existence of a $(w,v)$-hole indicates that the hypergraph does not have $(w,u)$-property. 
\begin{lem}[Lemma 2.1 in~\cite{frankl2021local}]\label{lemma:hole}
    For integers $q> p\ge r$, let $\cH$ be an $r$-graph with $(q,p)$-property. If $\cH$ has a $(w,v)$-hole $W$ such that $w\le q-r$, then the induced sub-$r$-graph of $\cH$ on $V(\cH)-W$, $\cH\cap\binom{V(\cH)-W}{r}$, has $(q-w,p-v)$-property. 
\end{lem}

Since $\t_r(\gamma)$ is non-increasing, it suffices to show the lower bounds for rational numbers $\gamma=s/t$. We make use of the following lemma:

\begin{lem}[Generalization of Lemma 2.2 in~\cite{frankl2021local}]\label{lemma:dense}
    For integers $s> t$, $r\ge2$, and $l\ge1$. Suppose an $r$-graph $\mathcal{F}$ on $sl$ vertices has $(q,p)$-property for all pairs $(q,p)$ with $q\le sl$ and $p=\lceil tq/s \rceil$. Then
    \begin{equation*}
        |\mathcal{F}|\ge \frac{s}{t}\binom{tl}{r}.
    \end{equation*}
\end{lem}

\begin{proof}
    Let $g(i):=\lceil ti/s\rceil$ and let $f(i):=\binom{g(i)-1}{r-1}$. We claim that any induced sub-$r$-graph of $\mathcal{F}$ on $i$ vertices must contain a vertex with degree at least $f(i)$. This is because $\mathcal{F}$ has $(i,g(i))$-property. Hence any induced sub-$r$-graph of $\mathcal{F}$ on $i$ vertices contains a clique of size $g(i)$, and hence contains a vertex with degree at least $\binom{g(i)-1}{r-1}=f(i)$. Therefore, 
    $$
    |\mathcal{F}|\ge \sum_{i=1}^{sl}f(i).
    $$
    Suppose $s=at+b$ for some integers $a$ and $0\le b<t$. It is not hard to check that $|g^{-1}(j)|=a$ or $a+1$. Let $\delta(j):=|g^{-1}(j)|-a$. So we have
    $$
    \sum_{i=1}^{sl}f(i)=a\sum_{j=1}^{tl}\binom{j-1}{r-1}+\sum_{j=1}^{tl}\delta(j)\binom{j-1}{r-1}=a\binom{tl}{r}+\sum_{j=1}^{tl}\delta(j)\binom{j-1}{r-1}.
    $$
    Now it suffices to show
    $$
    \sum_{j=1}^{tl}\delta(j)\binom{j-1}{r-1}\ge \frac{b}{t}\binom{tl}{r}=\frac{b}{t}\sum_{j=1}^{tl}\binom{j-1}{r-1}.
    $$
    We prove the following stronger statement.
    \begin{claim}
    For integer $1\le k\le tl$,
    $$
    \sum_{j=tl-k+1}^{tl}\l(\delta(j)-\frac{b}{t}\r)\binom{j-1}{r-1}\ge \sum_{j=tl-k+1}^{tl}\l(\delta(j)-\frac{b}{t}\r)\binom{tl-k}{r-1}.     
    $$
    \end{claim}
    \begin{proof}
        When $k=1$, the inequality is true. Suppose the inequality holds for $1\le k<tl$. Note that for $1\le k\le tl$,
        $$\sum_{j=tl-k+1}^{tl}\delta(j)=\l\lceil\frac{ks}{t}\r\rceil-ka.$$
        We have
        $$
        \begin{aligned}
        \sum_{j=tl-k}^{tl}\l(\delta(j)-\frac{b}{t}\r)\binom{j-1}{r-1}&\ge \sum_{j=tl-k+1}^{tl}\l(\delta(j)-\frac{b}{t}\r)\binom{tl-k}{r-1}+\l(\delta(tl-k)-\frac{b}{t}\r)\binom{tl-k-1}{r-1}\\
        &=\l(\l\lceil\frac{ks}{t}\r\rceil-\frac{ks}{t}\r)\binom{tl-k}{r-1}+\l(\delta(tl-k)-\frac{b}{t}\r)\binom{tl-k-1}{r-1}\\
        &\ge \l(\l\lceil\frac{ks}{t}\r\rceil-\frac{ks}{t}\r)\binom{tl-k-1}{r-1}+\l(\delta(tl-k)-\frac{b}{t}\r)\binom{tl-k-1}{r-1}\\
        &=\sum_{j=tl-k}^{tl}\l(\delta(j)-\frac{b}{t}\r)\binom{tl-k-1}{r-1}.
        \end{aligned}
        $$
        This completes the proof of the claim.
    \end{proof}
    The lemma follows immediately by putting $k=tl$ in the claim.
    
\end{proof}

We generalize the notion of {\em excess} from~\cite{frankl2021local}. For a pair $(q,p)$ with $tq\le sp$, we call $e(q,p)=sp-tq$ the excess of the pair $(q,p)$. Note that since $q\ge p$, we always have $e(q,p)\le (s-t)q$. In the following proof, a $(w,v)$-hole always has $sv<tw$ and $v\ge r-1$.

\begin{proof}[Proof of Theorem~\ref{theorem:general_lower_bounds}]
    Given $\epsilon>0$, fix $l$ such that
    $$
    \frac{{s}\binom{tl}{r}}{t\binom{sl}{r}}>(1-\epsilon/2)\l(\frac{t}{s}\r)^{r-1}.
    $$
    Then fix a much larger integer $L\ge3t^2l^2$. Let $\cH_0$ be an $r$-graph on $n$ vertices with $(sL,tL)$-property, where $n$ is sufficiently large. Our goal is to apply Lemma~\ref{lemma:dense} on a subset $X\subset [n]$ with $\binom{|X|}{r}\ge(1-\epsilon/2)\binom{n}{r}$. This requires us to find such an $X$ such that $\binom{X}{r}\cap \cH_0$ has no $(w,v)$-hole with $w\le sl$ and $v<tw/s$.

    We start with $\cH_0$ and define $\cH_i$ inductively. Let $q_0=sL$ and $p_0=tL$, $X_0=[n]$. Suppose $q_i\ge 2sl$ and $\cH_i=\binom{X_i}{r}\cap \cH_0$ has a $(w,v)$-hole $Y$ such that $w\le sl$. Let $q_{i+1}=q_i-w$, $p_{i+1}=p_i-v$, and let $X_{i+1}=X_i\setminus{Y}$, $\cH_{i+1}=\binom{X_{i+1}}{r}\cap\cH_0$.
    Note that $e(q_{i+1},p_{i+1})=s(p_i-v)-t(q_i-w)\ge e(q_i,p_i)+1$. This implies that $i\le e(q_i,p_i)\le (s-t)q_i$.
    
    The process keeps going unless, at some step $i$, $q_i<2sl$ or $\cH_i$ contains no $(w,v)$-holes with $w\le sl$. 

    If, at step $i$, $q_i<2sl$, then we have
    $sL=q_0\le q_i+isl\le q_i+(s-t)q_isl<3s^3l$. This contradicts $L\ge 3s^2l^2$.

    Hence, at some step $i$, we obtain $\cH_i$ containing no $(w,v)$-holes with $w\le sl$ such that $|X_i|\ge n-isl\ge n-(s-t)q_0sl$. For $n$ sufficiently large, 
    we have $\binom{|X_i|}{r}\ge(1-\epsilon/2)\binom{n}{r}$. By Lemma~\ref{lemma:dense}, 
    $$
    \frac{|\cH_0|}{\binom{n}{r}}\ge (1-\epsilon/2)\frac{|\cH_i|}{\binom{|X_i|}{r}}\ge (1-\epsilon/2)\frac{s\binom{tl}{r}}{t\binom{sl}{r}}\ge (1-\epsilon/2)^2\l(\frac{t}{s}\r)^{r-1}.
    $$
    Let $n\to\infty$ and then let $\epsilon\to0$. This completes the proof.
\end{proof}

\section{Some better lower bounds}\label{section:better}
In this section, we prove Theorem~\ref{theorem:better_lower_bound}. We make use of the following lemma:
\begin{lem}\label{lemma:recursive_lower_bound}
    For any integers $r\ge2$, $q_1\ge p_1\ge r$, $q_2\ge p_2\ge r$,
    \begin{equation*}
        t_r(q_1+q_2, p_1+p_2-1)\ge\min\{t_r(q_1,p_1),~t_r(q_2,p_2)\}
    \end{equation*}
\end{lem}

\begin{proof}
    Let $q_0=q_1+q_2$, $p_0=p_1+p_2-1$ and let $x=\min\{t_r(q_1,p_1),~t_r(q_2,p_2)\}$. For any $\epsilon>0$, let $n_0$ be sufficiently large such that for any $n\ge n_0$,
    $$
    \min\{t_r(n,q_1,p_1),t_r(n,q_2,p_2)\}\ge (1-\epsilon)x,
    $$
    and
    $$
    \binom{n-q_1}{r}\ge (1-\epsilon)\binom{n}{r}.
    $$
    For any $n\ge n_0+q_1$, let $G$ be an $r$-graph on $n$ vertices with $(q_0,p_0)$-property. If $G$ has $(q_1,p_1)$-property, then we have
    $$
    \frac{|\G|}{\binom{n}{r}}\ge t_r(n,q_1,p_1)\ge (1-\epsilon)x;
    $$
    
    Otherwise, $\G$ has a $(q_1,u)$-hole $Y$ where $u\le p_1-1$. Let $\G'=\binom{V(\G)\setminus Y}{r}\cap G$. By Lemma~\ref{lemma:hole}, $\G'$ has $(q_2,p_2)$-property. Therefore,
    $$
    \frac{|\G|}{\binom{n}{r}}\ge (1-\epsilon)\frac{|\G'|}{\binom{n-q_1}{r}}\ge (1-\epsilon)t_r(n-q_1, q_2, p_2)\ge(1-\epsilon)^2x.
    $$
    Let $n\to\infty$ and let $\epsilon\to 0$, we infer $t_r(q_0,p_0)\ge x$.
\end{proof}

\begin{lem}\label{lemma:degenerate}
    For any integers $q\ge p\ge r\ge 2$, $k\ge 0$,
    \begin{equation*}
        t_r(kq+q,k(p-1)+p)\ge t_r(q,p).
    \end{equation*}
\end{lem}

\begin{proof}
    We use induction on $k$. When $k=0$, equality holds. When $k\ge 1$, by Lemma~\ref{lemma:recursive_lower_bound} and the inductive hypothesis
    $$
    t_r(kq+q,k(p-1)+p)\ge \min\{t_r((k-1)q+q,(k-1)(p-1)+p),t_r(q,p)\}\ge t_r(q,p).
    $$
\end{proof}

\begin{proof}[Proof of Theorem~\ref{theorem:better_lower_bound}]
    Let $p=\ceil{q/\gamma}$, then
    $$
    \lim_{k\to\infty}\frac{kq+q}{k(p-1)+p}=\frac{q}{p-1}>\gamma.
    $$
    Hence there exists $k_0$ such that, for any $k\ge k_0$, $(kq+q)/\gamma\ge k(p-1)+p$. Therefore, by Lemma~\ref{lemma:degenerate}
    $$
    \t_r(\gamma)= \lim_{k\to\infty}t_r(kq+q,\ceil{(kq+q)/\gamma})\ge t_r(kq+q,k(p-1)+p)\ge t_r(q,p).
    $$

    Now suppose that $\t_r(\alpha)=t_r(q,\ceil{q/\alpha})$. For any $\alpha\le\gamma<\frac{q}{\ceil{q/\alpha}-1}$, we have $\ceil{q/\alpha}-1< \frac{q}{\gamma}\le\frac{q}{\alpha}$, hence $\ceil{q/\gamma}=\ceil{q/\alpha}$,  $\t_r(\gamma)\ge t_r(q,\ceil{q/\gamma})=t_r(q,\ceil{q/\alpha})=\t_r(\alpha)$. On the other hand, since the function $\t_r$ is non-increasing, we have $\t_r(\gamma)\le\t_r(\alpha)$. Therefore, $\t_r(\gamma)=t_r(q,p)$.
\end{proof}

Corollary~\ref{corollary:7/3} and Corollary~\ref{corollary:turan} follow easily from Theorem~\ref{theorem:better_lower_bound}.

\section{Constructions for upper bounds}\label{section:tight}
In this section, we describe a framework that generates constructions for upper bounds of $\t_r(\gamma)$ and then use it to prove Theorem~\ref{theorem:tightcycle_construction}. We will also briefly introduce the constructions used by Fang, Gao, Ma and Song~\cite{Fang2023Onlocal}.

\begin{definition}\label{definition:blowup}
    Fix a hypergraph $\cH$ on $[t]=\{1,\dots,t\}$. Let $\F(\cH,r)$ be a family of $r$-graphs $\G$ with the following properties: 
    vertices of $\G$ can be partitioned into $t$ disjoint sets $V_1,\dots,V_t$ such that 
    \begin{equation}\label{equation:blowup}
        \G=\cup_{e\in \cH}\binom{\cup_{i\in e}V_i}{r}.
    \end{equation}
\end{definition}

For example, a 3-graph $\G\in\F(C_3,3)$ consists of all triples $e$ such that $|e\cap V_i|\ge 2$ for some $i\in[3]$. Let
\begin{equation*}
    \rho(\cH,r,n)=\min\l\{\frac{|\G|}{\binom{n}{r}}:~\G\in\F(\cH,r),~|V(\G)|=n\r\},
\end{equation*}
and let
\begin{equation*}
    \rho(\cH,r)=\lim_{n\rightarrow\infty}\rho(\cH,r,n).
\end{equation*}

By an averaging argument, we know that, when $\cH$ and $r$ are fixed, $\rho(\cH,r,n)$ is non-decreasing. Hence, the limit $\rho(\cH,r)$ exists. 

Next, we introduce two parameters of a hypergraph $\cH$ that will make it more convenient to describe properties of the family $\F(\cH,r)$. For every integer $t\ge1$, let 
\begin{equation*}
\Delta_{t-1}=\{(x_1,\dots,x_t)|\forall 1\le i\le t,~x_i\ge0,~x_1+\dots+x_t=1\}.
\end{equation*} 
Given a hypergraph $\cH$ on $[t]$, we define a function $f_{\cH}$ on $\Delta_{t-1}$, 
\begin{equation*}
    f_{\cH}(x_1,\dots,x_t)=\max_{e\in\cH}\sum_{i\in e}x_i.
\end{equation*}

\begin{definition}
    Given a hypergraph $\cH$ on $[t]$, let
    \begin{equation*}
        f(\cH)=\min_{x\in \Delta_{t-1}} f_{\cH}(x).
    \end{equation*}
\end{definition}

Since $f_\cH$ is continuous and $\Delta_{t-1}$ is compact, $f(\cH)$ is well-defined. The intuition is, given a vertex set $A$ in a hypergraph $\cH\in \F(\cH,r)$, we can let $(x_1,\dots,x_t)$ be the vertex distribution of $A$ among different parts of $\cH$. Now if $f_\cH$ is large, then no matter how the vertices in $A$ are distributed, we can always find a clique in $A$ of size $f_\cH |A|$. The next lemma reveals the relation between $f_\cH$ and the local Tur\'an property of a hypergraph in $\rho(\cH,r,n)$ in a more formal way.

\begin{lem}\label{lemma:f}
    Let $\cH$ be a hypergraph on $[t]$ and let $\gamma=1/f(\cH)$. Then $\forall q\ge r\ge 2$ and $\forall \G\in\F(\cH,r)$, $\G$ has $(q,\ceil{q/\gamma})$-property. Therefore,
    $$
    \t_r(\gamma)\le \rho(\cH,r).
    $$
\end{lem}

\begin{proof}
    Let $V_1,\dots,V_t$ be disjoint sets of vertices of $\G$ satisfying (\ref{equation:blowup}). Let $W$ be a set of $q$ vertices of $\G$. Let $a_i=|W\cap V_i|$, $\forall 1\le i\le t$. Clearly, $(a_1/q,\dots,a_t/q)$ is a vector in $\Delta_{t-1}$. Hence by definition, there exists an edge $e\in \cH$ such that
    $$
    \frac{\sum_{i\in e}a_i}{q}=f_{\cH}(\frac{a_1}{q},\dots,\frac{a_t}{q})\ge f(\cH) \ge\frac{1}{\gamma}.
    $$
    Since $\sum_{i\in e}a_i$ is an integer, we have $\sum_{i\in e}a_i\ge\ceil{q/\gamma}$. By equation (\ref{equation:blowup}), $W\cap(\cup_{i\in e}V_i)$ spans a clique of size $\sum_{i\in e}a_i\ge\ceil{q/\gamma}$ in $\G$. Therefore, $\G$ has $(q,\ceil{q/\gamma})$-property.

    As a result, we have $t_r(q,\ceil{q/\gamma},n)\le \rho(\cH,r,n)$. Letting $n\to\infty$ and $q\to\infty$ gives $\t_r(\gamma)\le\rho(\cH,r)$.
\end{proof}

A hypergraph $\cH$ is said to be the disjoint union of $k$ hypergraphs $\cH_1,\dots,\cH_k$ where $k\ge2$ if $\cH_1,\dots,\cH_k$ have disjoint vertex sets and $\cH=\cH_1\cup\dots\cup\cH_k$.
\begin{lem}\label{lemma:f_union}
    Let $k\ge2$ and let $\cH$ be the disjoint union of $k$ hypergraphs $\cH_1,
    \dots,\cH_k$, then
    \begin{equation*}
        f(\cH)=\l(\sum_{i=1}^k f(\cH_k)^{-1}\r)^{-1}
    \end{equation*}
\end{lem}

\begin{proof}
    It suffices to show this for $k=2$. When $k=2$, by definition, 
    \begin{equation*}
        f(\cH)=\min_{0\le c\le 1}\{\max\{cf(\cH_1),(1-c)f(\cH_2)\}\}.
    \end{equation*}
    Clearly, the minimum is obtained when $cf(\cH_1)=(1-c)f(\cH_2)$, which implies
    \begin{equation*}
        c=\frac{f(\cH_2)}{f(\cH_1)+f(\cH_2)}.
    \end{equation*}
    Hence
    \begin{equation*}
        f(\cH)=\l(f(\cH_1)^{-1}+f(\cH_2)^{-1}\r)^{-1}.
    \end{equation*}
\end{proof}

Given a hypergraph $\cH$ on $[t]$ and $r\ge2$, we define a function $g_{\cH,r}$ on $\Delta_{t-1}$,
\begin{equation*}
    g_{\cH,r}(x_1,\dots,x_t)=\sum_{(y_1,\dots,y_r)\in [t]^r}\prod^r_{i=1}x_{y_i}\sigma_{\cH}(y_1,\dots,y_r),
\end{equation*}
where $\sigma_{\cH}(y_1,\dots,y_r)=1$ if there exists $e\in\cH$ such that $y_i\in e,~\forall 1\le i\le r$; otherwise, $\sigma(y_1,\dots,y_r)=0$.

\begin{definition}
    Given a hypergraph $\cH$ on $[t]$ and an integer $r\ge2$, let
    \begin{equation*}
        g(\cH,r)=\min_{x\in \Delta_{t-1}} g_{\cH,r}(x).
    \end{equation*}
\end{definition}

Since $g_{\cH,r}$ is continuous and $\Delta_{t-1}$ is compact, $g(\cH,r)$ is well-defined.

\begin{lem}\label{lemma:g_union}
    Let $r,k\ge2$ and let $\cH$ be the disjoint union of $k$ hypergraphs $\cH_1, \dots,\cH_k$, then
    \begin{equation*}
        g(\cH,r)=\l(\sum_{i=1}^k g(\cH_k,r)^{-\frac{1}{r-1}}\r)^{-r+1}.
    \end{equation*}
\end{lem}

\begin{proof}
    Again, it is sufficient to show the statement for $k=2$. When $k=2$, by definition, 
    \begin{equation*}
        g(\cH,r)=\min_{0\le c\le 1}\{c^rg(\cH_1,r)+(1-c)^rg(\cH_2,r)\}.
    \end{equation*}
    Let $g(\cH_1,r)=g_1$, $g(\cH_2,r)=g_2$ and let $b=(g_1/g_2)^{\frac{1}{r-1}}$. Applying Jensen's inequality to the convex function $x^r$,
    \begin{equation*}
    \begin{aligned}
        c^rg_1+(1-c)^rg_2&=c^rg_1+\l(\frac{1-c}{b}\r)^rb^rg_2\\
        &\ge (g_1+b^rg_2)\l(\frac{cg_1+(1-c)b^{r-1}g_2}{g_1+b^rg_2}\r)^r\\
        &=(g_1^{-\frac{1}{r-1}}+g_2^{-\frac{1}{r-1}})^{-r+1}
    \end{aligned}
    \end{equation*}
    Equality holds if and only if $c=(1-c)/b$ which is equivalent to $c=1/(1+b)$. Hence this minimum is obtainable.
    Therefore,
    \begin{equation*}
        g(\cH,r)=\l(g(\cH_1,r)^{-\frac{1}{r-1}}+g(\cH_2,r)^{-\frac{1}{r-1}}\r)^{-r+1}.
    \end{equation*}
\end{proof}

\begin{lem}\label{lemma:g}
    Let $\cH$ be a hypergraph on $[t]$ and $r\ge2$. Then
    \begin{equation*}
        g(\cH,r)=\rho(\cH,r).
    \end{equation*}
\end{lem}

\begin{proof}
    For any $\G\in\F(\cH,r)$, let $V_1,\dots,V_t$ be disjoint set of vertices of $\G$ satisfying (\ref{equation:blowup}). Given $y=(y_1,\dots,y_r)\in[t]^r$, $\forall 1\le i\le t$, let $m_y(i)$ denote the number of entries equal to $i$ in $y$. By (\ref{equation:blowup}), we can count the number of edges in $\G$ as following: for each $y=(y_1,\dots,y_r)\in [t]^r$, if $\sigma_\cH(y)=1$, then the number of (ordered) edges $e=(v_1,\dots,v_r)$ such that $v_i\in V_{y_i}$ for all $1\le i\le r$ is exactly
    $$
    \prod_{i=1}^t\binom{|V_i|}{m_y(i)}{m_y(i)\,!}.
    $$
    If we sum the quantity above over all $y\in[t]^r$, then each edge in $\G$ is counted exactly $r\,!$ times. Hence,
    \begin{equation*}
        |\G|=\sum_{y\in [t]^r}\sigma_\cH(y)\prod_{i=1}^t\binom{|V_i|}{m_y(i)}\frac{m_y(i)\,!}{r\,!}
        \ge\l(1-\frac{c_1}{v(\G)}\r)g_{\cH,r}\l(\frac{|V_1|}{v(\G)},\dots,\frac{|V_t|}{v(\G)}\r)\binom{v(\G)}{r}
    \end{equation*}
    where $c_1$ is a constant that depends only on $r$.
    This implies that
    \begin{equation*}
        \rho(\cH,r,n)\ge\l(1-\frac{c_1}{n}\r) g_{\cH,r}\l(\frac{|V_1|}{v(\G)},\dots,\frac{|V_t|}{v(\G)}\r)\ge\l(1-\frac{c_1}{n}\r) g(\cH,r).
    \end{equation*}
    Letting $n\to\infty$ gives $\rho(\cH,r)\ge g(\cH,r)$.

    On the other hand, let $x=(x_1,\dots,x_t)\in \Delta_{t-1}$ be the vector such that $g_{\cH,r}(x)=g(\cH,r)$. We can pick a sequence of $r$-graphs $\{\G_n\}_{n\ge 1}\subset\F(\cH,r)$ such that $v(\G_n)=n$ and $|V_i(\G_n)|/n\to x_i$ as $n\to\infty$, $\forall 1\le i\le t$. By definition, there exists a constant $c_2$ that depends only on $r$ such that
    \begin{equation*}
        g_{\cH,r}\l(\frac{|V_1(\G_n)|}{n},\dots,\frac{|V_t(\G_n|}{n}\r)\ge\l(1-\frac{c_2}{n}\r)\frac{|\G_n|}{\binom{n}{r}}\ge \l(1-\frac{c_2}{n}\r)\rho(\cH,r,n).
    \end{equation*}
    Letting $n\to\infty$ gives $g(\cH,r)\ge\rho(\cH,r)$.
    
\end{proof}

\begin{definition}\label{definition:tightcycle}
    For integers $t>m\ge2$, an $m$-uniform tight cycle of length $t$, denoted by $C^m_t$ is an $m$-graph on $[t]$ with $t$ edges $e_i=\{i,~{i+1},~\dots,~{i+m-1}\}$, $1\le i\le t$ (vertices are modulo $t$).
\end{definition}

We will use tight cycles of the forms $C^{m+1}_{2m+1}$ and $C^{m}_{m-1}$ as building blocks to give constructions that prove Theorem~\ref{theorem:tightcycle_construction}. Next, we introduce two lemmas that compute the parameters $f$ and $g$ of tight cycles.

\begin{lem}\label{lemma:f_tightcycle}
    For every inters $t>m\ge 2$, 
\begin{equation*}
    f(C^m_t)=\frac{m}{t}.
\end{equation*}
\end{lem}

\begin{proof}
    We will use the notations in Definition~\ref{definition:tightcycle}. For any $x\in \Delta_{t-1}$, let $x=(x_1,\dots,x_t)$. Then 
   \begin{equation*}
       t\cdot f_{C^m_t}(x)\ge \sum_{i=1}^t\sum_{i\in e_t}x_i=m\sum_{i=1}^tx_i=m.
   \end{equation*}
    Hence, $f(C^m_t)\ge m/t$. On the other hand, if we let $x=(1/t,\dots,1/t)$, then $f_{C^m_t}(x)=m/t$. Therefore, $f(C^m_t)=m/t$.
\end{proof}

\begin{lem}\label{lemma:g_tightcycle}
    Let $r\ge 3$ be an integer.
    \begin{itemize}
    \item[(1)] For every integer $m\ge 1$,
    \begin{equation*}
        g(C^{m+1}_{2m+1},r)\le\frac{(m+1)^r-m^r}{(2m+1)^{r-1}}.
    \end{equation*}
    \item[(2)] For every integer $3\le m\le r$,
    \begin{equation*}
        g(C^{m-1}_{m},r)\le\sum_{k=1}^{m-1}(-1)^{k+1}\binom{m}{k}\l(1-\frac{k}{m}\r)^r.
    \end{equation*}
\end{itemize}
\end{lem}

\begin{proof}
     \begin{itemize}
        \item [(1)] Throughout this proof the indices in $[2m+1]$ are modulo $2m+1$. For any $x=(x_1,\dots,x_{2m+1})\in X_{2m+1}$, we claim that
        \begin{equation}\label{equation:m+1_2m+1}
            g_{C^{m+1}_{2m+1},r}(x)=\sum_{i=1}^{2m+1}(x_i+\dots+x_{i+m})^r-\sum_{i=1}^{2m+1}(x_i+\dots+x_{i+m-1})^r.
        \end{equation}
        For any $y=(y_1,\dots,y_{r})\in[2m+1]^r$ such that $\sigma_{C^{m+1}_{2m+1}}(y)=1$, it suffices to show that the coefficient of the ordered monomial $x_{y_1}\dots x_{y_r}$ is $1$ in the right-hand side of (\ref{equation:m+1_2m+1}). Let $j$ be the minimum integer such that there exists $a\in [2m+1]$ such that $y_i\in\{a,\dots,a+j\}$, $\forall 1\le i\le r$. By definition, such $j$ must exist and satisfy $0\le j\le m$. Then the ordered monomial $x_{y_1}\dots x_{y_r}$ is contained in $(x_i+\dots+x_{i+m})^r$ for $i\in \{a+j-m,\dots,a\}$, which contributes $m-j+1$ to its coefficient; and is contained in $(x_i+\dots+x_{i+m-1})^r$ for $i\in \{a+j-m+1,\dots,a\}$, which contributes $-(m-j)$ to its coefficient. Hence its coefficient is exactly 1. This proves the claim.

        Now let $x=(1/(2m+1),\dots,1/(2m+1))$. We have
        \begin{equation*}
            g(C^{m+1}_{2m+1},r)\le g_{C^{m+1}_{2m+1},r}(x)=\frac{(m+1)^r-m^r}{(2m+1)^{r-1}}.
        \end{equation*}

        \item[(2)] For $x=(x_1,\dots,x_{m})\in X_{m}$, note that the polynomial $g_{C^{m-1}_{m},r}(x)$ is the sum of all ordered monomials of degree $r$ that is missing some $x_i$, each of them has coefficient 1. For any subset $S\in [m]$, observe that $(\sum_{i=1}^m x_i-\sum_{i\in S} x_i)^r$  is the sum of all ordered monomials of degree $r$ that is missing all $x_i$ with $i\in S$, each of them has coefficient 1. Hence by the Inclusion-Exclusion Principle, we have
        \begin{equation*}
            g_{C^{m-1}_{m},r}(x)=\sum_{k=1}^{m-1}\sum_{S\subset[m],|S|=k}(-1)^{k+1}(\sum_{i=1}^m x_i-\sum_{i\in S} x_i)^r.
        \end{equation*}

        Now let $x=(1/(2m+1),\dots,1/(2m+1))$. We have
        \begin{equation*}
            g(C^{m-1}_{m},r)\le g_{C^{m-1}_{m},r}(1/m,\dots,1/m)=\sum_{k=1}^{m-1}(-1)^{k+1}\binom{m}{k}\l(1-\frac{k}{m}\r)^r.
        \end{equation*}
    \end{itemize}
\end{proof}

Let $I_k$ be the $1$-uniform hypergraph with $k$ vertices and $k$ edges. Given a hypergraph $\cH$, let $\cH\cup I_k$ denote the disjoint union of $\cH$ and $I_k$. It is easy to check that $f(I_k)=1/k$ and $g(I_k,r)=k^{-r+1}$.

\begin{proof}[Proof of Theorem~\ref{theorem:tightcycle_construction}]
    By Lemma~\ref{lemma:f_tightcycle} and Lemma~\ref{lemma:f_union},
    \begin{equation*}
        f(C^m_t\cup I_k)=\l(k+\frac{t}{m}\r)^{-1}.
    \end{equation*}
    \begin{enumerate}
        \item[(1)]
         By Lemma~\ref{lemma:f}, Lemma~\ref{lemma:g} and the fact that $\t_r$ is non-increasing, when $\gamma\ge k+\frac{m}{m+1}=1/f(C^{m+1}_{2m+1}\cup I_{k-1})$, 
        \begin{equation*}
            \t_r(\gamma)\le g(C^{m+1}_{2m+1}\cup I_{k-1}).
        \end{equation*}
        
        By Lemma~\ref{lemma:g_tightcycle} and Lemma~\ref{lemma:g_union},
        \begin{equation*}
            g(C^{m+1}_{2m+1}\cup I_{k-1})\le\l(k-1+(2m+1)((m+1)^r-m^r)^{-\frac{1}{r-1}}\r)^{-r+1}.
        \end{equation*}
        Therefore,
        \begin{equation*}
            \t_r(\gamma)\le\l(k-1+(2m+1)((m+1)^r-m^r)^{-\frac{1}{r-1}}\r)^{-r+1}.
        \end{equation*}
       
        \item[(2)]
        By Lemma~\ref{lemma:f}, Lemma~\ref{lemma:g} and the fact that $\t_r$ is non-increasing, when $\gamma\ge k+\frac{1}{m-1}=1/f(C^{m-1}_{m}\cup I_{k-1})$, 
        \begin{equation*}
            \t_r(\gamma)\le g(C^{m-1}_{m}\cup I_{k-1}).
        \end{equation*}
        
        By Lemma~\ref{lemma:g_tightcycle} and Lemma~\ref{lemma:g_union},
        \begin{equation*}
            g(C^{m-1}_{m}\cup I_{k-1})\le\l(k-1+\l(\sum_{i=1}^{m-1}(-1)^{i+1}\binom{m}{i}\l(1-\frac{i}{m}\r)^r\r)^{-\frac{1}{r-1}}\r)^{-r+1}.
        \end{equation*}
        Therefore,
        \begin{equation*}
            \t_r(\gamma)\le\l(k-1+\l(\sum_{i=1}^{m-1}(-1)^{i+1}\binom{m}{i}\l(1-\frac{i}{m}\r)^r\r)^{-\frac{1}{r-1}}\r)^{-r+1}.
        \end{equation*}
    \end{enumerate}
\end{proof}

Fang, Gao, Ma and Song~\cite{Fang2023Onlocal} make use of the following hypergraphs.
\begin{definition}
    For integer $m\ge2$, let $W_m$ be a hypergraph on $[m+1]$ defined as following:
    \begin{equation*}
        W_m=\cup_{i=1}^m\{\{i,m+1\}\}\cup\{\{1,\dots,m\}\}.
    \end{equation*}
\end{definition}

It is not hard to check that $f(W_m)=m/(2m-1)$ and $g(W_m,3)=(5m+4)/9m$. Hence, by Lemma~\ref{lemma:f}, Lemma~\ref{lemma:f_union}, Lemma~\ref{lemma:g} and Lemma~\ref{lemma:g_union}, for every integers $k,m\ge1$ and real number $\gamma\ge k+\frac{m}{m+1}$, 
\begin{equation*}
    \t_3(\gamma)\le\l(k-1+3\sqrt{\frac{m+1}{5m+9}}\r)^{-2}.
\end{equation*}
This is strictly better than (\ref{equation:upper_bounds}) when $r=3$ and $m\ge 2$.

\section{Concluding remarks}
\begin{itemize}
    \item By Lemma~\ref{lemma:f} and Lemma~\ref{lemma:g} we know that
    \begin{equation}\label{equation:general_upperbound}
        \t_r(\gamma)\le \min\{g(\cH,r)|f(\cH)\ge 1/\gamma\}.
    \end{equation}

    Is this bound tight? All known correct constructions so far support this intuition. To begin with, can we show that $\F(C^2_3,r)$ provides the best constructions for $\t_r(\frac{3}{2})$? In other words, is it true that
    $$
    \t_r\l(\frac{3}{2}\r)=\frac{2^r-1}{3^{r-1}}?
    $$
    The case $r=3$ has been confirmed recently by Fang, Gao, Ma and Song~\cite{Fang2023Onlocal}. 
    \item For $0\le\epsilon<1$, we know from the proof of Theorem~\ref{theorem:tightcycle_construction} that $\t_r(k+\epsilon)\le \l(k-1+\t_r(1+\epsilon)^{-\frac{1}{r-1}}\r)^{-r+1}$. Is this tight? In other words, is it true that the best construction for $\t_r(k+\epsilon)$ comes from taking the disjoint union of the best construction for $\t_r(1+\epsilon)$ and $(k-1)$ cliques? This is true for $\epsilon=0$ by Theorem~\ref{theorem:Frankl-Huang-Rodl}.
\end{itemize}

\section*{Aknowledgement}
We would like to thank Hao Huang, Jie Ma and Hehui Wu for helpful discussions.


\bibliographystyle{abbrv}
\bibliography{refs}

\section*{Appendix A}
\begin{prop}
    For every integer $r\ge2$ and real number $\gamma>1$, the limits
    \begin{equation*}
        \t_r(\gamma)=\lim_{q\to\infty}t_r(q,\ceil{q/\gamma})
    \end{equation*}
    exist.
\end{prop}

\begin{proof}
    Let $T=\limsup_{q\to\infty}t_r(q,\ceil{q/\gamma})$. Let $\epsilon>0$ be an arbitrarily small real number. We need to show that there exists $q_0=q_0(\epsilon)$ such that for all $q\ge q_0$, $t_r(q,\ceil{q/\gamma})\ge T-\epsilon$. By definition, there exists $q'$ such that $t_r(q',p')\ge T-\epsilon$ where $p'=\ceil{q'/\gamma}$. Clearly, $q'/(p'-1)>\gamma$. Let $\delta=\frac{1}{\gamma}-(p'-1)/q'$. Let $q_0$ be large enough such that for every $q\ge q_0$
    \begin{equation}\label{equation:condition1}
        \frac{\ceil{q/\gamma}-1}{q}\ge \frac{1}{\gamma}-\frac{\delta}{2}
    \end{equation}
    
    and
    \begin{equation}\label{equation:condition2}
        \frac{\delta q}{2(1-\frac{1}{\gamma}+\delta)}\ge q'.
    \end{equation}
    Fix a $q\ge q_0$, and let $p=\ceil{q/\gamma}$. For $k\ge 1$, let $q_{k}=q-kq'$, $p_k=p-k(p'-1)$. By Lemma~\ref{lemma:recursive_lower_bound}, as long as $t_r(q_k,p_k)$ is well defined,
    $$
    t_r(q,p)\ge \min\{t_r(q',p'),t_r(q_{k},p_{k})\}.
    $$
    
    Let $k'$ be the unique integer such that $0\le q_{k'}-p_{k'}\le q'-p'$. Not hard to check that
    $$
    k'=\l\lceil\frac{q-p+1}{q'-p'+1}\r\rceil-1.
    $$
    By (\ref{equation:condition1}), we have
    $$
    k'\le\frac{q(1-\frac{1}{\gamma}+\frac{\delta}{2})}{q'(1-\frac{1}{\gamma}+\delta)}.
    $$
    Hence, by (\ref{equation:condition2}),
    $$
    q_{k'}=q-k'q'\ge \frac{\delta q}{2(1-\frac{1}{\gamma}+\delta)}\ge q'.
    $$
    Note that by definition, $t_r(a,b)\ge t_r(a-1,b-1)\ge t_r(a, b-1)$. So we have,
    $$
    t_r(q_{k'},p_{k'})\ge t_r(q',p').
    $$
    Therefore,
    $$
    t_r(q,p)\ge t_r(q',p')\ge T-\epsilon.
    $$
\end{proof}

\end{document}